\documentclass{amsart}
\usepackage{amsfonts}
\usepackage{amssymb}
\usepackage{amsmath}

\newtheorem{theorem}{Theorem}[section]
\theoremstyle{plain}

\newtheorem{corollary}[theorem]{Corollary}

\newtheorem{definition}{Definition}[section]

\newtheorem{lemma}{Lemma}[section]

\newtheorem{remark}[theorem]{Remark}

\numberwithin{equation}{section}
\input{tcilatex}

\begin{document}
\title[On neighborhood and partial sums problem]{On neighborhood and partial
sums problem for generalized Sakaguchi type functions}
\author{Murat \c{C}A\u{G}LAR and Halit ORHAN}
\address{Department of Mathematics, Faculty of Science, Ataturk University,
Erzurum, 25240, Turkey.}
\email{mcaglar25@gmail.com, horhan607@gmail.com}

\begin{abstract}
In the present investigation, we introduce a new class $k-\widetilde{%
\mathcal{US}}_{s}^{\eta }(\lambda ,\mu ,\gamma ,t)$ of analytic functions in
the open unit disc $\mathcal{U}$ with negative coefficients. The object of
the present paper is to determine coefficient estimates, neighborhoods and
partial sums for functions $f(z)$ belonging to this class.
\end{abstract}

\maketitle

\section{\textbf{Introduction}}

Let $\mathcal{A}$ denote the family of functions $f$ of the form%
\begin{equation}
f(z)=z+\sum\limits_{n=2}^{\infty }{a_{n}z^{n}}  \label{eq1}
\end{equation}%
that are analytic in the open unit disk $\mathcal{U}=\left\{ {z:\left\vert
z\right\vert <1}\right\} $. Denote by $\mathcal{S}$ the subclass of $%
\mathcal{A}$ of functions that are univalent in $\mathcal{U}$.

For $f\in \mathcal{A}$ given by (\ref{eq1}) and $g(z)$ given by%
\begin{equation}
g(z)=z+\sum\limits_{n=2}^{\infty }b{_{n}z^{n}}\text{ \ \ }  \label{eq1*}
\end{equation}%
their convolution (or Hadamard product), denoted by $(f\ast g),$ is defined
as%
\begin{equation}
(f\ast g)(z):=z+\sum\limits_{n=2}^{\infty }a_{n}b{_{n}z^{n}=:(g\ast f)(z)}%
\text{ \ \ }\left( z\in \mathcal{U}\right) .  \label{eq1**}
\end{equation}%
Note that $f\ast g\in $ $\mathcal{A}.$

A function $f\in \mathcal{A}$ is said to be in $k-\mathcal{US(\gamma )},$
the class of $k-$uniformly starlike functions of order $\gamma ,$ $0\leq
\gamma <1,$ if satisfies the condition%
\begin{equation}
\func{Re}\left\{ \frac{zf^{\prime }(z)}{f(z)}\right\} >k\left\vert \frac{%
zf^{\prime }(z)}{f(z)}-1\right\vert +\gamma \text{ \ \ }\left( k\geq
0\right) ,  \label{eq2}
\end{equation}%
and a function $f\in \mathcal{A}$ is said to be in $k-\mathcal{UC(\gamma )},$
the class of $k-$uniformly convex functions of order $\gamma ,$ $0\leq
\gamma <1,$ if satisfies the condition%
\begin{equation}
\func{Re}\left\{ 1+\frac{zf^{\prime \prime }(z)}{f^{\prime }(z)}\right\}
>k\left\vert \frac{zf^{\prime \prime }(z)}{f^{\prime }(z)}\right\vert
+\gamma \text{ \ \ }\left( k\geq 0\right) .  \label{eq3}
\end{equation}%
Uniformly starlike and uniformly convex functions were first introduced by
Goodman \cite{Good} and then studied by various authors. It is known that $%
f\in k-\mathcal{UC(\gamma )}$ or $f\in k-\mathcal{US(\gamma )}$ if and only
if $1+\frac{zf^{\prime \prime }(z)}{f^{\prime }(z)}$ or $\frac{zf^{\prime
}(z)}{f(z)},$ respectively, takes all the values in the conic domain $%
\mathcal{R}_{k,\gamma }$ which is included in the right half plane given by%
\begin{equation}
\mathcal{R}_{k,\gamma }:=\left\{ w=u+iv\in 
\mathbb{C}
:u>k\sqrt{\left( u-1\right) ^{2}+v^{2}}+\gamma ,\text{ }\beta \geq 0\text{
and }\gamma \in \left[ 0,1\right) \right\} .  \label{eq33}
\end{equation}

Denote by $\mathcal{P(}P_{k,\gamma }),$ $(\beta \geq 0,$ $0\leq \gamma <1)$
the family of functions $p,$ such that $p\in \mathcal{P},$ where $\mathcal{P}
$ denotes well-known class of Caratheodory functions. The function $%
P_{k,\gamma }$ maps the unit disk conformally onto the domain $\mathcal{R}%
_{k,\gamma }$ such that $1\in \mathcal{R}_{k,\gamma }$ and $\partial 
\mathcal{R}_{k,\gamma }$ is a curve defined by the equality%
\begin{equation}
\partial \mathcal{R}_{k,\gamma }:=\left\{ w=u+iv\in 
\mathbb{C}
:u^{2}=\left( k\sqrt{\left( u-1\right) ^{2}+v^{2}}+\gamma \right) ^{2},\text{
}\beta \geq 0\text{ and }\gamma \in \left[ 0,1\right) \right\} .
\label{eq333}
\end{equation}

From elementary computations we see that (\ref{eq333}) represents conic
sections symmetric about the real axis. Thus $\mathcal{R}_{k,\gamma }$ is an
elliptic domain for $k>1,$ a parabolic domain for $k=1,$ a hyperbolic domain
for $0<k<1$ and the right half plane $u>\gamma ,$ for $\beta =0.$

In \cite{Sak}, Sakaguchi defined the class $\mathcal{S}_{s}$ of starlike
functions with respect to symmetric points as follows:

Let $f\in \mathcal{A}$. Then $f$ is said to be starlike with respect to
symmetric points in $\mathcal{U}$ if and only if%
\begin{equation*}
\func{Re}\left\{ \frac{2zf^{\prime }(z)}{f(z)-f(-z)}\right\} >0\text{ \ \ }%
\left( z\in \mathcal{U}\right) .
\end{equation*}%
Recently, Owa et. al. \cite{Owa} defined the class $\mathcal{S}_{s}(\alpha
,t)$ as follows:%
\begin{equation*}
\func{Re}\left\{ \frac{(1-t)zf^{\prime }(z)}{f(z)-f(tz)}\right\} >\alpha 
\text{ \ \ }\left( z\in \mathcal{U}\right) ,
\end{equation*}%
where $0\leq \alpha <1,$ $\left\vert t\right\vert \leq 1,$ $t\neq 1.$ Note
that $\mathcal{S}_{s}(0,-1)=\mathcal{S}_{s}$ and $\mathcal{S}_{s}(\alpha
,-1)=\mathcal{S}_{s}(\alpha )$ is called Sakaguchi function of order $\alpha
.$

The \textit{linear multiplier differential operator }$D_{\lambda ,\mu
}^{\eta }f$ was defined by the authors in (see \cite{RaOr}) as follows%
\begin{eqnarray*}
D_{\lambda ,\mu }^{0}f(z) &=&f(z) \\
D_{\lambda ,\mu }^{1}f(z) &=&D_{\lambda ,\mu }f(z)=\lambda \mu
z^{2}(f(z))^{\prime \prime }+(\lambda -\mu )z(f(z))^{\prime }+(1-\lambda
+\mu )f(z) \\
D_{\lambda ,\mu }^{2}f(z) &=&D_{\lambda ,\mu }\left( D_{\lambda ,\mu
}^{1}f(z)\right) \\
&&\vdots \\
D_{\lambda ,\mu }^{\eta }f(z) &=&D_{\lambda ,\mu }\left( D_{\lambda ,\mu
}^{\eta -1}f(z)\right)
\end{eqnarray*}

\noindent where $0\leqslant \mu \leqslant \lambda \leqslant 1$ and $\eta \in 
\mathbb{N}_{0}=\mathbb{N}\cup \{0\}.$ Later, the operator $D_{\lambda ,\mu
}^{\eta }f$ was extended for $\lambda \geqslant \mu \geqslant 0$ by the
authors in (see \cite{DeOr}).

If $f$ is given by (\ref{eq1}) then from the definition of the operator $%
D_{\lambda ,\mu }^{\eta }f(z)$ it is easy to see that%
\begin{equation}
D_{\lambda ,\mu }^{\eta }f(z)=z+\sum\limits_{n=2}^{\infty }\Phi ^{\eta
}\left( \lambda ,\mu ,n\right) a_{n}z^{n}  \label{eq4}
\end{equation}%
where%
\begin{equation}
\Phi ^{\eta }\left( \lambda ,\mu ,n\right) ={[1+(\lambda \mu n+\lambda -\mu
)(n-1)]^{\eta }.}  \label{eq5}
\end{equation}

It should be remarked that the operator $D_{\lambda ,\mu }^{\eta }$ is a
generalization of many other linear operators considered earlier. In
particular, for $f\in \mathcal{A}$ we have the following:

\begin{itemize}
\item $D_{1,0}^{\eta }f(z)\equiv D^{\eta }f(z)$ the operator investigated by
S\u{a}l\u{a}gean (see \cite{Sa}).

\item $D_{\lambda ,0}^{\eta }f(z)\equiv D_{\lambda }^{\eta }f(z)$ the
operator studied by Al-Oboudi (see \cite{Al}).
\end{itemize}

Now, by making use of the differential operator $D_{\lambda ,\mu }^{\eta },$
we define a new subclass of functions belonging to the class $\mathcal{A}$ $%
. $

\begin{definition}
\label{d1} A function $f(z)\in \mathcal{A}$ is said to be in the class $k-%
\mathcal{US}_{s}^{\eta }(\lambda ,\mu ,\gamma ,t)$ if for all $z\in \mathcal{%
U},$%
\begin{equation*}
\func{Re}\left\{ \frac{(1-t)z\left( D_{\lambda ,\mu }^{\eta }f(z)\right)
^{\prime }}{D_{\lambda ,\mu }^{\eta }f(z)-D_{\lambda ,\mu }^{\eta }f(tz)}%
\right\} \geq k\left\vert \frac{(1-t)z\left( D_{\lambda ,\mu }^{\eta
}f(z)\right) ^{\prime }}{D_{\lambda ,\mu }^{\eta }f(z)-D_{\lambda ,\mu
}^{\eta }f(tz)}-1\right\vert +\gamma
\end{equation*}%
for $\lambda \geqslant \mu \geqslant 0,$ $\eta ,k\geq 0,$ $\left\vert
t\right\vert \leq 1,$ $t\neq 1,$ $0\leq \gamma <1.$
\end{definition}

Furthermore, we say that a function $f(z)\in k-\mathcal{US}_{s}^{\eta
}(\lambda ,\mu ,\gamma ,t)$ is in the subclass $k-\widetilde{\mathcal{US}}%
_{s}^{\eta }(\lambda ,\mu ,\gamma ,t)$ if $f(z)$ is of the following form:%
\begin{equation}
f(z)=z-\sum\limits_{n=2}^{\infty }a_{n}{z^{n}}\text{ \ \ }({a_{n}\geq 0,}%
\text{ }n\in 
\mathbb{N}
).  \label{eq7}
\end{equation}

The aim of the present paper is to study the coefficient bounds, partial
sums and certain neighborhood results of the class $k-\widetilde{\mathcal{US}%
}_{s}^{\eta }(\lambda ,\mu ,\gamma ,t).$

\begin{remark}
\label{r1} Throught our present investigation, we tacitly assume that the
parametric constraints listed (\ref{eq5}).
\end{remark}

\section{\textbf{Coefficient bounds of the function class }$k-\widetilde{%
\mathcal{US}}_{s}^{\protect\eta }(\protect\lambda ,\protect\mu ,\protect%
\gamma ,t)$}

Firstly, we shall need the following lemmas.

\begin{lemma}
\label{l1} Let $w=u+iv.$ Then%
\begin{equation*}
\func{Re}w\geq \alpha \text{ \ if and only if \ }\left\vert w-(1+\alpha
)\right\vert \leq \left\vert w+(1-\alpha )\right\vert .
\end{equation*}
\end{lemma}

\begin{lemma}
\label{l2} Let $w=u+iv$ and $\alpha ,\gamma $ are real numbers. Then%
\begin{equation*}
\func{Re}w>\alpha \left\vert w-1\right\vert +\gamma \text{ \ if and only if
\ }\func{Re}\left\{ w\left( 1+\alpha e^{i\theta }\right) -\alpha e^{i\theta
}\right\} >\gamma .
\end{equation*}
\end{lemma}

\begin{theorem}
\label{t1} The function $f(z)$ defined by (\ref{eq7}) is in the class $k-%
\widetilde{\mathcal{US}}_{s}^{\eta }(\lambda ,\mu ,\gamma ,t)$ if and only if%
\begin{equation}
\dsum\limits_{n=2}^{\infty }\Phi ^{\eta }\left( \lambda ,\mu ,n\right)
\left\vert n(k+1)-u_{n}\left( k+\gamma \right) \right\vert a_{n}\leq
1-\gamma ,  \label{eq8}
\end{equation}%
where $\lambda \geqslant \mu \geqslant 0,$ $\eta ,k\geq 0,$ $\left\vert
t\right\vert \leq 1,$ $t\neq 1,$ $0\leq \gamma <1,$ $u_{n}=1+t+...+t^{n-1}.$

The result is sharp for the function $f(z)$ given by%
\begin{equation*}
f(z)=z-\frac{1-\gamma }{\Phi ^{\eta }\left( \lambda ,\mu ,n\right)
\left\vert n(k+1)-u_{n}\left( k+\gamma \right) \right\vert }z^{n}.
\end{equation*}
\end{theorem}

\begin{proof}
By Definition \ref{d1}, we get%
\begin{equation*}
\func{Re}\left\{ \frac{(1-t)z\left( D_{\lambda ,\mu }^{\eta }f(z)\right)
^{\prime }}{D_{\lambda ,\mu }^{\eta }f(z)-D_{\lambda ,\mu }^{\eta }f(tz)}%
\right\} \geq k\left\vert \frac{(1-t)z\left( D_{\lambda ,\mu }^{\eta
}f(z)\right) ^{\prime }}{D_{\lambda ,\mu }^{\eta }f(z)-D_{\lambda ,\mu
}^{\eta }f(tz)}-1\right\vert +\gamma .
\end{equation*}%
Then by Lemma \ref{l2}, we have%
\begin{equation*}
\func{Re}\left\{ \frac{(1-t)z\left( D_{\lambda ,\mu }^{\eta }f(z)\right)
^{\prime }}{D_{\lambda ,\mu }^{\eta }f(z)-D_{\lambda ,\mu }^{\eta }f(tz)}%
\left( 1+ke^{i\theta }\right) -ke^{i\theta }\right\} \geq \gamma ,\text{ \ \ 
}-\pi <\theta \leq \pi
\end{equation*}%
or equivalently%
\begin{equation}
\func{Re}\left\{ \frac{(1-t)z\left( D_{\lambda ,\mu }^{\eta }f(z)\right)
^{\prime }\left( 1+ke^{i\theta }\right) }{D_{\lambda ,\mu }^{\eta
}f(z)-D_{\lambda ,\mu }^{\eta }f(tz)}-\frac{ke^{i\theta }\left[ D_{\lambda
,\mu }^{\eta }f(z)-D_{\lambda ,\mu }^{\eta }f(tz)\right] }{D_{\lambda ,\mu
}^{\eta }f(z)-D_{\lambda ,\mu }^{\eta }f(tz)}\right\} \geq \gamma .
\label{eq9}
\end{equation}%
Let%
\begin{equation*}
F(z)=(1-t)z\left( D_{\lambda ,\mu }^{\eta }f(z)\right) ^{\prime }\left(
1+ke^{i\theta }\right) -ke^{i\theta }\left[ D_{\lambda ,\mu }^{\eta
}f(z)-D_{\lambda ,\mu }^{\eta }f(tz)\right]
\end{equation*}%
and%
\begin{equation*}
E(z)=D_{\lambda ,\mu }^{\eta }f(z)-D_{\lambda ,\mu }^{\eta }f(tz).
\end{equation*}%
By Lemma \ref{l1}, (\ref{eq9}) is equivalent to 
\begin{equation*}
\left\vert F(z)+(1-\gamma )E(z)\right\vert \geq \left\vert F(z)-(1+\gamma
)E(z)\right\vert \text{ \ for \ }0\leq \gamma <1.
\end{equation*}%
But%
\begin{eqnarray*}
\left\vert F(z)+(1-\gamma )E(z)\right\vert &=&\left\vert \left( 1-t\right)
\left\{ \left( 2-\gamma \right) z-\dsum\limits_{n=2}^{\infty }\Phi ^{\eta
}\left( \lambda ,\mu ,n\right) (n+u_{n}\left( 1-\gamma \right)
)a_{n}z^{n}\right. \right. \\
&&\left. \left. -ke^{i\theta }\dsum\limits_{n=2}^{\infty }\Phi ^{\eta
}\left( \lambda ,\mu ,n\right) \left( n-u_{n}\right) a_{n}z^{n}\right\}
\right\vert \\
&\geq &\left\vert 1-t\right\vert \left\{ \left( 2-\gamma \right) \left\vert
z\right\vert -\dsum\limits_{n=2}^{\infty }\Phi ^{\eta }\left( \lambda ,\mu
,n\right) \left\vert n+u_{n}\left( 1-\gamma \right) \right\vert
a_{n}\left\vert z\right\vert ^{n}\right. \\
&&\left. -k\dsum\limits_{n=2}^{\infty }\Phi ^{\eta }\left( \lambda ,\mu
,n\right) \left\vert n-u_{n}\right\vert a_{n}\left\vert z\right\vert
^{n}\right\} .
\end{eqnarray*}%
Also%
\begin{eqnarray*}
\left\vert F(z)-(1+\gamma )E(z)\right\vert &=&\left\vert \left( 1-t\right)
\left\{ -\gamma z-\dsum\limits_{n=2}^{\infty }\Phi ^{\eta }\left( \lambda
,\mu ,n\right) (n-\left( 1+\gamma \right) u_{n})a_{n}z^{n}\right. \right. \\
&&\left. \left\vert -ke^{i\theta }\dsum\limits_{n=2}^{\infty }\Phi ^{\eta
}\left( \lambda ,\mu ,n\right) \left( n-u_{n}\right) a_{n}z^{n}\right\}
\right\vert \\
&\leq &\left\vert 1-t\right\vert \left\{ \gamma \left\vert z\right\vert
+\dsum\limits_{n=2}^{\infty }\Phi ^{\eta }\left( \lambda ,\mu ,n\right)
\left\vert n-u_{n}\left( 1+\gamma \right) \right\vert a_{n}\left\vert
z\right\vert ^{n}\right. \\
&&\left. +k\dsum\limits_{n=2}^{\infty }\Phi ^{\eta }\left( \lambda ,\mu
,n\right) \left\vert n-u_{n}\right\vert a_{n}\left\vert z\right\vert
^{n}\right\}
\end{eqnarray*}%
and so 
\begin{eqnarray*}
&&\left\vert F(z)+(1-\gamma )E(z)\right\vert -\left\vert F(z)-(1+\gamma
)E(z)\right\vert \\
&\geq &\left\vert 1-t\right\vert \left\{ 2(1-\gamma )\left\vert z\right\vert
-\dsum\limits_{n=2}^{\infty }\Phi ^{\eta }\left( \lambda ,\mu ,n\right) 
\left[ \left\vert n+u_{n}\left( 1-\gamma \right) \right\vert +\left\vert
n-u_{n}\left( 1+\gamma \right) \right\vert +2k\left\vert n-u_{n}\right\vert %
\right] a_{n}\left\vert z\right\vert ^{n}\right\} \\
&\geq &2(1-\gamma )\left\vert z\right\vert -\dsum\limits_{n=2}^{\infty
}2\Phi ^{\eta }\left( \lambda ,\mu ,n\right) \left\vert n(k+1)-u_{n}\left(
k+\gamma \right) \right\vert a_{n}\left\vert z\right\vert ^{n}\geq 0
\end{eqnarray*}%
or%
\begin{equation*}
\dsum\limits_{n=2}^{\infty }\Phi ^{\eta }\left( \lambda ,\mu ,n\right)
\left\vert n(k+1)-\left( k+\gamma \right) u_{n}\right\vert a_{n}\leq
1-\gamma .
\end{equation*}%
Conversely, suppose that (\ref{eq8}) holds. Then we must show%
\begin{equation*}
\func{Re}\left\{ \frac{(1-t)z\left( D_{\lambda ,\mu }^{\eta }f(z)\right)
^{\prime }\left( 1+ke^{i\theta }\right) -ke^{i\theta }\left[ D_{\lambda ,\mu
}^{\eta }f(z)-D_{\lambda ,\mu }^{\eta }f(tz)\right] }{D_{\lambda ,\mu
}^{\eta }f(z)-D_{\lambda ,\mu }^{\eta }f(tz)}\right\} \geq \gamma .
\end{equation*}%
Upon choosing the values of $z$ on the positive real axis where $0\leq
z=r<1, $ the above inequality reduces to%
\begin{equation*}
\func{Re}\left\{ \frac{\left( 1-\gamma \right) -\dsum\limits_{n=2}^{\infty
}\Phi ^{\eta }\left( \lambda ,\mu ,n\right) \left( n(1+ke^{i\theta
})-u_{n}\left( \gamma +ke^{i\theta }\right) \right) a_{n}z^{n-1}}{%
1-\dsum\limits_{n=2}^{\infty }\Phi ^{\eta }\left( \lambda ,\mu ,n\right)
u_{n}a_{n}z^{n-1}}\right\} \geq 0.
\end{equation*}%
Since $\func{Re}(-e^{i\theta })\geq -\left\vert e^{i\theta }\right\vert =-1,$
the above inequality reduces to%
\begin{equation*}
\func{Re}\left\{ \frac{\left( 1-\gamma \right) -\dsum\limits_{n=2}^{\infty
}\Phi ^{\eta }\left( \lambda ,\mu ,n\right) \left( n(1+k)-u_{n}\left( \gamma
+k\right) \right) a_{n}r^{n-1}}{1-\dsum\limits_{n=2}^{\infty }\Phi ^{\eta
}\left( \lambda ,\mu ,n\right) u_{n}a_{n}r^{n-1}}\right\} \geq 0.
\end{equation*}%
Letting $r\rightarrow 1^{-},$ we have desired conclusion.
\end{proof}

\begin{corollary}
\label{c1} If $f(z)\in k-\widetilde{\mathcal{US}}_{s}^{\eta }(\lambda ,\mu
,\gamma ,t),$ then%
\begin{equation*}
a_{n}\leq \frac{1-\gamma }{\Phi ^{\eta }\left( \lambda ,\mu ,n\right)
\left\vert n(k+1)-u_{n}\left( k+\gamma \right) \right\vert }
\end{equation*}%
where $\lambda \geqslant \mu \geqslant 0,$ $\eta ,k\geq 0,$ $\left\vert
t\right\vert \leq 1,$ $t\neq 1,$ $0\leq \gamma <1,$ $u_{n}=1+t+...+t^{n-1}.$
\end{corollary}

\section{\textbf{Neighborhood of the function class }$k-\widetilde{\mathcal{%
US}}_{s}^{\protect\eta }(\protect\lambda ,\protect\mu ,\protect\gamma ,t)$}

Following the earlier investigations (based upon the familiar concept of
neighborhoods of analytic functions) by Goodman \cite{Good1}, Ruscheweyh 
\cite{Rusc}, Alt\i nta\c{s} \textit{et al. }(\cite{Altin1}, \cite{Altin})
and others including Srivastava \textit{et al.} (\cite{SriAo}, \cite{SriOw}%
), Orhan (\cite{Orhan1}), Deniz \textit{et al. }\cite{DeOr1}, Cata\c{s} \cite%
{Cat}.

\begin{definition}
\label{d2} Let $\lambda \geqslant \mu \geqslant 0,$ $\eta ,k\geq 0,$ $%
\left\vert t\right\vert \leq 1,$ $t\neq 1,$ $0\leq \gamma <1,$ $\alpha \geq
0,$ $u_{n}=1+t+...+t^{n-1}.$ We define the $\alpha -$neighborhood of a
function $f\in \mathcal{A}$ and denote by $\mathcal{N}_{\alpha }(f)$
consisting of all functions $g(z)=z-\dsum\limits_{n=2}^{\infty
}b_{n}z^{n}\in \mathcal{S}$ $(b{_{n}\geq 0,}$ $n\in 
\mathbb{N}
)$ satisfying%
\begin{equation*}
\dsum\limits_{n=2}^{\infty }\frac{\Phi ^{\eta }\left( \lambda ,\mu ,n\right)
\left\vert n(k+1)-u_{n}\left( k+\gamma \right) \right\vert }{1-\gamma }%
\left\vert a_{n}-b_{n}\right\vert \leq \alpha .
\end{equation*}
\end{definition}

\begin{theorem}
\label{t2} Let $f\in k-\widetilde{\mathcal{US}}_{s}^{\eta }(\lambda ,\mu
,\gamma ,t)$ and for all real $\theta $ we have $\gamma (e^{i\theta
}-1)-2e^{i\theta }\neq 0.$ For any complex number $\varepsilon $ with $%
\left\vert \varepsilon \right\vert <\alpha $ $\left( \alpha \geq 0\right) ,$
if $f$ satisfies the following condition:%
\begin{equation*}
\frac{f(z)+\varepsilon z}{1+\varepsilon }\in k-\widetilde{\mathcal{US}}%
_{s}^{\eta }(\lambda ,\mu ,\gamma ,t),
\end{equation*}%
then $\mathcal{N}_{\alpha }(f)\subset k-\widetilde{\mathcal{US}}_{s}^{\eta
}(\lambda ,\mu ,\gamma ,t).$
\end{theorem}

\begin{proof}
It is obvious that $f\in k-\widetilde{\mathcal{US}}_{s}^{\eta }(\lambda ,\mu
,\gamma ,t)$ if and only if%
\begin{equation*}
\left\vert \frac{(1-t)z\left( D_{\lambda ,\mu }^{\eta }f(z)\right) ^{\prime
}(1+ke^{i\theta })-(ke^{i\theta }+1+\gamma )\left( D_{\lambda ,\mu }^{\eta
}f(z)-D_{\lambda ,\mu }^{\eta }f(tz)\right) }{(1-t)z\left( D_{\lambda ,\mu
}^{\eta }f(z)\right) ^{\prime }(1+ke^{i\theta })+(1-ke^{i\theta }-\gamma
)\left( D_{\lambda ,\mu }^{\eta }f(z)-D_{\lambda ,\mu }^{\eta }f(tz)\right) }%
\right\vert <1\text{ \ \ }\left( -\pi <\theta <\pi \right)
\end{equation*}%
for any complex number $s$ with $\left\vert s\right\vert =1,$ we have%
\begin{equation*}
\frac{(1-t)z\left( D_{\lambda ,\mu }^{\eta }f(z)\right) ^{\prime
}(1+ke^{i\theta })-(ke^{i\theta }+1+\gamma )\left( D_{\lambda ,\mu }^{\eta
}f(z)-D_{\lambda ,\mu }^{\eta }f(tz)\right) }{(1-t)z\left( D_{\lambda ,\mu
}^{\eta }f(z)\right) ^{\prime }(1+ke^{i\theta })+(1-ke^{i\theta }-\gamma
)\left( D_{\lambda ,\mu }^{\eta }f(z)-D_{\lambda ,\mu }^{\eta }f(tz)\right) }%
\neq s.
\end{equation*}%
In other words, we must have%
\begin{equation*}
(1-s)(1-t)z\left( D_{\lambda ,\mu }^{\eta }f(z)\right) ^{\prime
}(1+ke^{i\theta })-(ke^{i\theta }+1+\gamma +s(ke^{i\theta }-1+\gamma )\left(
D_{\lambda ,\mu }^{\eta }f(z)-D_{\lambda ,\mu }^{\eta }f(tz)\right) \neq 0
\end{equation*}%
which is equivalent to%
\begin{equation*}
z-\dsum\limits_{n=2}^{\infty }\frac{\Phi ^{\eta }\left( \lambda ,\mu
,n\right) \left( (n-u_{n})(1+ke^{i\theta }-ske^{i\theta }\right)
-s(n+u_{n})-u_{n}\gamma (1-s))}{\gamma (s-1)-2s}z^{n}\neq 0.
\end{equation*}%
However, $f\in k-\widetilde{\mathcal{US}}_{s}^{\eta }(\lambda ,\mu ,\gamma
,t)$ if and only if $\frac{(f\ast h)(z)}{z}\neq 0,$ $z\in \mathcal{U}-\{0\}$
where $h(z)=z-\dsum\limits_{n=2}^{\infty }c_{n}z^{n},$ and%
\begin{equation*}
c_{n}=\frac{\Phi ^{\eta }\left( \lambda ,\mu ,n\right) \left(
(n-u_{n})(1+ke^{i\theta }-ske^{i\theta }\right) -s(n+u_{n})-u_{n}\gamma
(1-s))}{\gamma (s-1)-2s}
\end{equation*}%
we note that 
\begin{equation*}
\left\vert c_{n}\right\vert \leq \frac{\Phi ^{\eta }\left( \lambda ,\mu
,n\right) \left\vert n(1+k)-u_{n}(k+\gamma )\right\vert }{1-\gamma }
\end{equation*}%
since $\frac{f(z)+\varepsilon z}{1+\varepsilon }\in k-\widetilde{\mathcal{US}%
}_{s}^{\eta }(\lambda ,\mu ,\gamma ,t),$ therefore $z^{-1}\left( \frac{%
f(z)+\varepsilon z}{1+\varepsilon }\ast h(z)\right) \neq 0,$ which is
equivalent to%
\begin{equation}
\frac{(f\ast h)(z)}{\left( 1+\varepsilon \right) z}+\frac{\varepsilon }{%
1+\varepsilon }\neq 0.  \label{eq10}
\end{equation}%
Now suppose that $\left\vert \frac{(f\ast h)(z)}{z}\right\vert <\alpha .$
Then by (\ref{eq10}), we must have%
\begin{equation*}
\left\vert \frac{(f\ast h)(z)}{\left( 1+\varepsilon \right) z}+\frac{%
\varepsilon }{1+\varepsilon }\right\vert \geq \frac{\left\vert \varepsilon
\right\vert }{\left\vert 1+\varepsilon \right\vert }-\frac{1}{\left\vert
1+\varepsilon \right\vert }\left\vert \frac{(f\ast h)(z)}{z}\right\vert >%
\frac{\left\vert \varepsilon \right\vert -\alpha }{\left\vert 1+\varepsilon
\right\vert }\geq 0,
\end{equation*}%
this is a contradiction by $\left\vert \varepsilon \right\vert <\alpha $ and
however, we have $\left\vert \frac{(f\ast h)(z)}{z}\right\vert \geq \alpha .$
If $g(z)=z-\dsum\limits_{n=2}^{\infty }b_{n}z^{n}\in \mathcal{N}_{\alpha
}(f),$ then%
\begin{eqnarray*}
\alpha -\left\vert \frac{(g\ast h)(z)}{z}\right\vert &\leq &\left\vert \frac{%
((f-g)\ast h)(z)}{z}\right\vert \leq \dsum\limits_{n=2}^{\infty }\left\vert
a_{n}-b_{n}\right\vert \left\vert c_{n}\right\vert \left\vert
z^{n}\right\vert \\
&<&\dsum\limits_{n=2}^{\infty }\frac{\Phi ^{\eta }\left( \lambda ,\mu
,n\right) \left\vert n(1+k)-u_{n}(k+\gamma )\right\vert }{1-\gamma }%
\left\vert a_{n}-b_{n}\right\vert \leq \alpha .
\end{eqnarray*}
\end{proof}

\section{\textbf{Partial sums of the function class }$k-\widetilde{\mathcal{%
US}}_{s}^{\protect\eta }(\protect\lambda ,\protect\mu ,\protect\gamma ,t)$}

In this section, applying methods used by Silverman \cite{Sil} and Silvia 
\cite{Silvia}, we investigate the ratio of a function of the form (\ref{eq7}%
) to its sequence of partial sums $f_{m}(z)=z+\sum\limits_{n=2}^{m}a_{n}{%
z^{n}}$.

\begin{theorem}
\label{t3}If $f$ of the form (\ref{eq1}) satisfies the condition (\ref{eq8}%
), then%
\begin{equation}
\func{Re}\left\{ \frac{f(z)}{f_{m}(z)}\right\} \geq 1-\frac{1}{\delta _{m+1}}
\label{eq11}
\end{equation}%
and%
\begin{equation}
\delta _{n}\geq \left\{ 
\begin{array}{cc}
1, & n=2,3,...,m \\ 
\delta _{m+1}, & n=m+1,m+2,...%
\end{array}%
\right.  \label{eq12}
\end{equation}%
where%
\begin{equation}
\delta _{n}=\frac{\Phi ^{\eta }\left( \lambda ,\mu ,n\right) \left\vert
n(k+1)-u_{n}\left( k+\gamma \right) \right\vert }{1-\gamma }.  \label{eq13}
\end{equation}%
The result in (\ref{eq11}) is sharp for every $m,$ with the extremal
function 
\begin{equation}
f(z)=z+\frac{z^{m+1}}{\delta _{m+1}}.  \label{eq14}
\end{equation}
\end{theorem}

\begin{proof}
Define the function $w(z)$, we may write%
\begin{eqnarray}
\frac{1+w(z)}{1-w(z)} &=&\delta _{m+1}\left\{ \frac{f(z)}{f_{m}(z)}-\left( 1-%
\frac{1}{\delta _{m+1}}\right) \right\}  \label{eq15} \\
&=&\left\{ \frac{1+\sum\limits_{n=2}^{m}a_{n}z^{n-1}+\delta
_{m+1}\sum\limits_{n=m+1}^{\infty }a_{n}z^{n-1}}{1+\sum%
\limits_{n=2}^{m}a_{n}z^{n-1}}\right\} .  \notag
\end{eqnarray}%
Then, from (\ref{eq15}) we can obtain%
\begin{equation*}
w(z)=\frac{\delta _{m+1}\sum\limits_{n=m+1}^{\infty }a_{n}z^{n-1}}{%
2+2\sum\limits_{n=2}^{m}a_{n}z^{n-1}+\delta
_{m+1}\sum\limits_{n=m+1}^{\infty }a_{n}z^{n-1}}
\end{equation*}%
and%
\begin{equation*}
\left\vert w(z)\right\vert \leq \frac{\delta
_{m+1}\sum\limits_{n=m+1}^{\infty }a_{n}}{2-2\sum\limits_{n=2}^{m}a_{n}-%
\delta _{m+1}\sum\limits_{n=m+1}^{\infty }a_{n}}.
\end{equation*}%
Now $\left\vert w(z)\right\vert \leq 1$ if%
\begin{equation*}
2\delta _{m+1}\sum\limits_{n=m+1}^{\infty }a_{n}\leq
2-2\sum\limits_{n=2}^{m}a_{n},
\end{equation*}%
which is equivalent to%
\begin{equation}
\sum\limits_{n=2}^{m}a_{n}+\delta _{m+1}\sum\limits_{n=m+1}^{\infty
}a_{n}\leq 1.  \label{eq16}
\end{equation}%
It is suffices to show that the left hand side of (\ref{eq16}) is bounded
above by $\sum\limits_{n=2}^{\infty }\delta _{n}a_{n},$ which is equivalent
to%
\begin{equation*}
\sum\limits_{n=2}^{m}\left( \delta _{n}-1\right)
a_{n}+\sum\limits_{n=m+1}^{\infty }\left( \delta _{n}-\delta _{m+1}\right)
a_{n}\geq 0.
\end{equation*}%
To see that the function given by (\ref{eq14}) gives the sharp result, we
observe that for $z=re^{i\pi /n},$%
\begin{equation}
\frac{f(z)}{f_{m}(z)}=1+\frac{z^{m}}{\delta _{m+1}}.  \label{eq17}
\end{equation}%
Taking $z\rightarrow 1^{-},$ we have%
\begin{equation*}
\frac{f(z)}{f_{m}(z)}=1-\frac{1}{\delta _{m+1}}.
\end{equation*}%
This completes the proof of Theorem \ref{t3}.
\end{proof}

We next determine bounds for $f_{m}(z)/f(z).$

\begin{theorem}
\label{t4}If $f$ of the form (\ref{eq1}) satisfies the condition (\ref{eq8}%
), then%
\begin{equation}
\func{Re}\left\{ \frac{f_{m}(z)}{f(z)}\right\} \geq \frac{\delta _{m+1}}{%
1+\delta _{m+1}}.  \label{eq18}
\end{equation}%
The result is sharp with the function given by (\ref{eq14}).
\end{theorem}

\begin{proof}
We may write%
\begin{eqnarray}
\frac{1+w(z)}{1-w(z)} &=&\left( 1+\delta _{m+1}\right) \left\{ \frac{f_{m}(z)%
}{f(z)}-\frac{\delta _{m+1}}{1+\delta _{m+1}}\right\}  \notag \\
&=&\left\{ \frac{1+\sum\limits_{n=2}^{m}a_{n}z^{n-1}-\delta
_{m+1}\sum\limits_{n=m+1}^{\infty }a_{n}z^{n-1}}{1+\sum\limits_{n=2}^{\infty
}a_{n}z^{n-1}}\right\} ,  \label{eq19}
\end{eqnarray}%
where%
\begin{equation*}
w(z)=\frac{\left( 1+\delta _{m+1}\right) \sum\limits_{n=m+1}^{\infty
}a_{n}z^{n-1}}{-\left( 2+2\sum\limits_{n=2}^{m}a_{n}z^{n-1}-\left( 1-\delta
_{m+1}\right) \sum\limits_{n=m+1}^{\infty }a_{n}z^{n-1}\right) },
\end{equation*}%
and%
\begin{equation}
\left\vert w(z)\right\vert \leq \frac{\left( 1+\delta _{m+1}\right)
\sum\limits_{n=m+1}^{\infty }a_{n}}{2-2\sum\limits_{n=2}^{m}a_{n}+\left(
1-\delta _{m+1}\right) \sum\limits_{n=m+1}^{\infty }a_{n}}\leq 1.
\label{eq20}
\end{equation}%
This last inequality is equivalent to%
\begin{equation}
\sum\limits_{n=2}^{m}a_{n}+\delta _{m+1}\sum\limits_{n=m+1}^{\infty
}a_{n}\leq 1.  \label{eq21}
\end{equation}%
It is suffices to show that the left hand side of (\ref{eq21}) is bounded
above by $\sum\limits_{n=2}^{\infty }\delta _{n}a_{n},$ which is equivalent
to%
\begin{equation*}
\sum\limits_{n=2}^{m}\left( \delta _{n}-1\right)
a_{n}+\sum\limits_{n=m+1}^{\infty }\left( \delta _{n}-\delta _{m+1}\right)
a_{n}\geq 0.
\end{equation*}%
This completes the proof of Theorem \ref{t4}.
\end{proof}

We next turn to ratios involving derivatives.

\begin{theorem}
\label{t5}If $f$ of the form (\ref{eq1}) satisfies the condition (\ref{eq8}%
), then%
\begin{equation}
\func{Re}\left\{ \frac{f^{\prime }(z)}{f_{m}^{\prime }(z)}\right\} \geq 1-%
\frac{m+1}{\delta _{m+1}},  \label{eq22*}
\end{equation}%
\begin{equation}
\func{Re}\left\{ \frac{f_{m}^{\prime }(z)}{f^{\prime }(z)}\right\} \geq 
\frac{\delta _{m+1}}{1+m+\delta _{m+1}}  \label{eq23}
\end{equation}%
where%
\begin{equation*}
\delta _{n}\geq \left\{ 
\begin{array}{cc}
1, & n=1,2,3,...,m \\ 
n\frac{\delta _{m+1}}{m+1}, & n=m+1,m+2,...%
\end{array}%
\right.
\end{equation*}

and $\delta _{n}$ is defined by (\ref{eq13}). The estimates in (\ref{eq22*})
and (\ref{eq23}) are sharp with the extremal function given by (\ref{eq14}).
\end{theorem}

\begin{proof}
Firstly, we will give proof of (\ref{eq22*}). We write%
\begin{eqnarray*}
\frac{1+w(z)}{1-w(z)} &=&\delta _{m+1}\left\{ \frac{f^{\prime }(z)}{%
f_{m}^{\prime }(z)}-\left( 1-\frac{1+m}{\delta _{m+1}}\right) \right\} \\
&=&\left\{ \frac{1+\sum\limits_{n=2}^{m}na_{n}z^{n-1}+\frac{\delta _{m+1}}{%
m+1}\sum\limits_{n=m+1}^{\infty }na_{n}z^{n-1}}{1+\sum%
\limits_{n=2}^{m}a_{n}z^{n-1}}\right\} ,
\end{eqnarray*}%
where%
\begin{equation*}
w(z)=\frac{\frac{\delta _{m+1}}{m+1}\sum\limits_{n=m+1}^{\infty
}na_{n}z^{n-1}}{2+2\sum\limits_{n=2}^{m}na_{n}z^{n-1}+\frac{\delta _{m+1}}{%
m+1}\sum\limits_{n=m+1}^{\infty }na_{n}z^{n-1}}
\end{equation*}%
and%
\begin{equation*}
\left\vert w(z)\right\vert \leq \frac{\frac{\delta _{m+1}}{m+1}%
\sum\limits_{n=m+1}^{\infty }na_{n}}{2-2\sum\limits_{n=2}^{m}na_{n}-\frac{%
\delta _{m+1}}{m+1}\sum\limits_{n=m+1}^{\infty }na_{n}}.
\end{equation*}%
Now $\left\vert w(z)\right\vert \leq 1$ if and only if%
\begin{equation}
\sum\limits_{n=2}^{m}na_{n}+\frac{\delta _{m+1}}{m+1}\sum\limits_{n=m+1}^{%
\infty }na_{n}\leq 1,  \label{eq24}
\end{equation}%
since the left hand side of (\ref{eq24}) is bounded above by is bounded
above by $\sum\limits_{n=2}^{\infty }\delta _{n}a_{n}.$

The proof of (\ref{eq23}) follows the pattern of that in Theorem (\ref{t4}).

This completes the proof of Theorem \ref{t5}.
\end{proof}

\end{document}